\documentclass{amsart}
\usepackage{a4}
\usepackage{enumerate}
\usepackage{leftidx}
\usepackage{mathtools}  % \coloneqq
\usepackage{slashbox}
\usepackage[arrow,matrix]{xy}

\DeclareMathOperator{\GL}{GL}
\DeclareMathOperator{\gal}{Gal}
\DeclareMathOperator{\en}{End}
\DeclareMathOperator{\res}{Res}
\DeclareMathOperator{\id}{id}
\DeclareMathOperator{\ext}{Ext}
\DeclareMathOperator{\Hom}{Hom}
\DeclareMathOperator{\im}{Im}
\DeclareMathOperator{\trg}{trg}
\DeclareMathOperator{\infl}{Inf}
\DeclareMathOperator{\Br}{Br}
\DeclareMathOperator{\inv}{inv}
\DeclareMathAlphabet{\mathpzc}{OT1}{pzc}{m}{it}

\newcommand{\isom}{\stackrel{\sim}{\longrightarrow}}

\theoremstyle{definition}
\newtheorem{definition}{Definition}[section]
\newtheorem{example}[definition]{Example}
\newtheorem{remark}[definition]{Remark}

\theoremstyle{plain}
\newtheorem{theorem}[definition]{Theorem}
\newtheorem{corollary}[definition]{Corollary}
\newtheorem{lemma}[definition]{Lemma}
\newtheorem{proposition}[definition]{Proposition}

\newcommand{\Z}{\mathbb{Z}}
\newcommand{\Q}{\mathbb{Q}}

\newcommand{\C}{\mathbb{C}}

\renewcommand{\H}{\mathrm{H}}
\newcommand{\rZ}{\mathrm{Z}}
\newcommand{\OQ}{\overline{\Q}}
\newcommand{\GQ}{G_{\Q}}
\newcommand{\nr}{\mathrm{nr}}
\newcommand{\s}{{}^\sigma\!}
\newcommand{\sym}{\mathrm{s}}

\newcommand{\n}{{}^\nu\!}

\renewcommand{\O}{\mathcal{O}}
\newcommand{\QA}{\mathbb{Q}^{\mathrm{ab}}}
\newcommand{\GQA}{G_{\Q}^{\mathrm{ab}}}
\newcommand{\vt}{{}^\vartheta\!}

\makeindex

\author[P.\thinspace J. Bruin]{Peter Bruin}
\address{Mathematisch Instituut\\
Universiteit Leiden\\
Postbus 9512\\
2300 RA \ Leiden\\
Netherlands\\
}
\email{P.J.Bruin@math.leidenuniv.nl}
\thanks{The first author was partially supported by the Netherlands Organisation for Scientific Research (NWO) through Veni grant 639.031.346}

\author[A. Ferraguti]{Andrea Ferraguti}
\address{University of Cambridge\\
DPMMS\\
Centre for Mathematical Sciences\\
Wilbeforce Road, Cambridge CB3 0WB, UK\\
}
\email{af612@dpmms.cam.ac.uk}
\thanks{The second author was supported by Swiss National Science Foundation grant number 168459.}
\date{}
\title{Strongly modular models of $\Q$-curves}

\keywords{$\Q$-curves; quadratic twists; strong modularity; Galois cohomology}

\begin{document}

\begin{abstract}
Let $E$ be a $\Q$-curve without complex multiplication. We address the problem of deciding whether $E$ is geometrically isomorphic to a strongly modular $\Q$-curve. We show that the question has a positive answer if and only if $E$ has a model that is completely defined over an abelian number field. Next, if $E$ is completely defined over a quadratic or biquadratic number field $L$, we classify all strongly modular twists of $E$ over~$L$ in terms of the arithmetic of~$L$. Moreover, we show how to determine which of these twists come, up to isogeny, from a subfield of~$L$.
\end{abstract}

\maketitle

\section{Introduction}

In the study of elliptic curves over number fields, the concept of modularity plays a central role. If $E$ is an elliptic curve over $\Q$ of conductor $N$, the modularity theorem \cite{bcdt,wil,tawil} states that there exists a non-trivial map of algebraic curves $X_0(N)\to E$, called a \emph{modular parametrization}, where $X_0(N)$ is the compact modular curve for $\Gamma_0(N)$. This fact has several important consequences, one of which is the fact that the $L$-function of $E$ coincides with the $L$-function of a weight $2$ newform of level $\Gamma_0(N)$. This in turn implies for example that the $L$-function has an analytic continuation to $\C$.

It is natural to ask for a generalization of this fact to elliptic curves over $\OQ$. Shimura proved \cite{shi2} that elliptic curves over $\OQ$ with complex multiplication (CM) admit a modular parametrization from the compact modular curve $X_1(N)$, for an appropriate positive integer~$N$. For curves without CM, the situation was more complicated. In 1992, Ribet \cite{rib1} proved that under Serre's modularity conjecture, an elliptic curve $E/\OQ$ without CM admits a modular parametrization from $X_1(N)$ if and only if $E$ is a $\Q$-curve, i.e.\ it is $\OQ$-isogenous to all of its Galois conjugates. The subsequent proof, in 2009, of the aforementioned conjecture by Khare and Wintenberger \cite{kwin} completed the characterization of this class of elliptic curves. Note that since there exists a natural map $X_1(N)\to X_0(N)$, the modularity theorem already shows that all elliptic curves over $\Q$ belong to this class. From now on, all $\Q$-curves will be implicitly assumed to be without CM.

Contrary to what happens for elliptic curves over $\Q$, the $L$-function $L(E/K,s)$ of a $\Q$-curve $E$ over a number field~$K\ne\Q$ is never the $L$-function of a newform. However, $L(E/K,s)$ may be a \emph{product} of $L$-functions of newforms. In fact, the proof of Ribet's theorem \cite{rib1} implies that abelian varieties of $\GL_2$-type are isogenous to products of abelian varieties attached to newforms of weight 2. For example, if $E$ is a $\Q$-curve over a quadratic field $K$ that is $K$-isogenous to its Galois conjugate, then its restriction of scalars to~$\Q$ is an abelian surface of $\GL_2$-type. Since $E$ and its restriction of scalars have the same $L$-function \cite{mil1}, it follows that $L(E/K,s)$ is a product of two $L$-functions of newforms of weight 2. This motivates the following definition \cite{guique}: a $\Q$-curve is said to be \emph{strongly modular} if its $L$-function is a product of $L$-functions of newforms of weight~2.

Guitart and Quer \cite{guique2,guique} gave necessary and sufficient conditions for a $\Q$-curve (and, more in general, for a building block) over a number field $K$ to be strongly modular: this happens if and only if $K/\Q$ is an abelian extension, $E$ is \emph{completely defined} over $K$ (i.e.\ all isogenies between conjugates of $E$ are defined over $K$) and the 2-cocycle attached to $E$ (cf.\ section \ref{strong_modularity}) is symmetric. It is easy to deduce from the results of \cite{gola}, \cite{gogui} and \cite{mil1} that every $\Q$-curve is geometrically isogenous to a strongly modular one. In this paper, we address the following question: what are necessary and sufficient conditions for a $\Q$-curve to admit a strongly modular model over some number field? Our main result is the following.
\begin{theorem}\label{main}
 A $\Q$-curve admits a strongly modular model if and only if it has a model completely defined over an abelian number field $K$.
\end{theorem}

The paper is structured as follows. In section \ref{group_cohomology} we review some basic facts about group cohomology and group extensions that will be needed later. In section \ref{strong_modularity} we recall the construction of two invariants attached to a $\Q$-curve $E$ over a Galois extension $K/\Q$ with Galois group $G$. Both invariants depend only on the $K$-isogeny class of $E$. The first one is a 1-cocycle for $G$ with values in $K^{\times}/{(K^{\times})}^2$ yielding information about the smallest field over which the curve is completely defined, which we call the \emph{minimal field of complete definition} (Definition \ref{minimal_field_def}). The second one is a 2-cocycle for $G$ with values in $(\en(E)\otimes\Q)^{\times}\simeq \Q^{\times}$ carrying information about the strong modularity of $E$ and the field of definition of $E$ up to isogeny (Theorem \ref{descent} and Proposition \ref{inflation}). Section \ref{strong_mod_iso} is dedicated to the proof of our main theorem. The proof relies on two preliminary results. The first one is \cite[Lemma 6.1]{guique}, which characterizes strongly modular twists $E^{\gamma}$ of a $\Q$-curve $E$ over a number field~$K$ in terms of the arithmetic of $K(\sqrt{\lambda})$. The second result is that the 2-torsion of the Brauer group of $\Q$ consists of inflations of symmetric 2-cocycles for $\gal(\QA/\Q)$ with values in $\{\pm1\}$. This essentially allows us to ``twist'' the 2-cocycle of a $\Q$-curve over an abelian number field into a symmetric one. Finally, in section \ref{section_twist} we study in detail an instance of the problem above: Theorem \ref{main} is not effective in general, but what about $\Q$-curves $E$ over a quadratic field $K$? We prove that every such curve has a model completely defined over a biquadratic, and hence abelian, number field $L$ containing $K$; it therefore admits a strongly modular model. We construct explicitly all strongly modular models of~$E$ over~$L$ (Theorem~\ref{primQcurves}), the existence of which is regulated only by the arithmetic of~$L$. We explain how to determine those models that descend to subfields of~$L$, up to isogeny. As a corollary, we show how to construct all strongly modular twists of~$E$ over~$K$. We end the section with several examples exhibiting different behaviours.

\section*{Notation and conventions}

When $A$ is an abelian variety over a field $K$ and $F$ is an extension of $K$, the $\Q$-algebra of the $F$-endomorphisms of $A$ is denoted by $\en_F^0(A)$.

If $G$ and $A$ are abelian groups, we write $\ext^1(G,A)$ for the group of abelian extensions of $G$ by~$A$.

If $F/K$ is a Galois extension of fields and $A$ is a $\gal(F/K)$-module, we denote by $\H^i(F/K,A)$ the $i$-th cohomology group of $A$ with coefficients in $\gal(F/K)$. Analogously, we denote by $\rZ^i(F/K,A)$ the group of $i$-cocycles. If $c\in \rZ^i(F/K,A)$, its cohomology class is denoted by $[c]$. If $G$ is a group and $A$ is a $G$-module, we denote the action of~$G$ on~$A$ using left superscripts: for $\sigma \in G$ and $a\in A$, the image of~$a$ under~$\sigma$ is written as $\s a$. We denote by $A^G$ the submodule of $G$-invariants.

All profinite groups that we mention throughout the paper are endowed with their profinite topology; in particular, finite groups are discrete.

We fix an algebraic closure $\OQ$ of $\Q$, and write $\QA$ for the maximal abelian extension of $\Q$ inside $\OQ$. The absolute Galois group of $\Q$ is denoted by $\GQ$ and the Galois group of $\QA$ over $\Q$ is denoted by $\GQA$.

For a unitary ring $R$, we denote by $R^{\times}$ the group of units.

\section{Group cohomology, group extensions and embedding problems}\label{group_cohomology}

In this section we collect some standard facts and definitions from the theory of group cohomology that we will use later in the paper. For a complete treatment of profinite group cohomology, see for example \cite{neu} or \cite{serre}.

Let $G$ be a profinite group, let $N$ be a normal closed subgroup of~$G$, and let $A$ be a $G$-module. For every $i\geq 1$, the \emph{restriction map} $\res\colon \H^i(G,A)\to \H^i(N,A)$ is the natural map induced by the inclusion $N\subseteq G$, and the \emph{inflation map} $\infl\colon \H^i(G/N,A^N)\to \H^i(G,A)$ is the natural map induced by the projection $G\to G/N$.

The group $\H^1(N,A)$ is endowed with an action of $G/N$ defined in the following way: for $g\in G$ and $[c]\in \H^1(N,A)$, let ${}^g[c]$ be the cohomology class represented by the cocycle $h\mapsto {}^{g\!}c(g^{-1}hg)$ for all $h\in N$. It is easy to check that $N$ acts as the identity, so that the action factors through the quotient $G/N$.
\begin{theorem}[{{\cite[Proposition I.1.6.5]{neu}}}]\label{trans}
	There is a natural map
	$$\trg\colon \H^1(N,A)^{G/N}\longrightarrow \H^2(G/N,A^N)$$
	fitting into an exact sequence
	$$\begin{aligned}
	0 & \longrightarrow \H^1(G/N,A^N) \stackrel{\infl}{\longrightarrow} \H^1(G,A) \stackrel{\res}{\longrightarrow} \H^1(N,A)^{G/N} \stackrel{\trg}{\longrightarrow} \\
	& \stackrel{\trg}{\longrightarrow} \H^2(G/N,A^N) \stackrel{\infl}{\longrightarrow} \H^2(G,A).
	\end{aligned}$$
\end{theorem}

The map $\trg$ in Theorem~\ref{trans} is called the \emph{transgression map}, and the exact sequence is called the \emph{inflation-restriction sequence}.

\begin{definition}
	Let $G$ be a profinite abelian group and let $A$ be an abelian group regarded as a $G$-module with trivial action.
	We say that a cocycle $c\in \rZ^2(G,A)$ is \emph{symmetric} if $c(\sigma,\tau)=c(\tau,\sigma)$ for all $\sigma,\tau\in G$.
\end{definition}
Note that the property of $c$ being symmetric only depends on the cohomology class of~$c$, because coboundaries are symmetric by the commutativity of $G$. Moreover, the product of two symmetric cocycles is clearly a symmetric cocycle. Thus the cohomology classes in $\H^2(G,A)$ represented by a symmetric cocycle form a subgroup of $\H^2(G,A)$, which we denote by $\H^2_\sym(G,A)$.

% TODO: better reference?
\begin{lemma}\label{extsym}
	Let $G$ be a finite abelian group, and let $A$ be an abelian group equipped with the trivial $G$-action. There is a canonical isomorphism of abelian groups
	$$\ext^1(G,A)\isom \H^2_\sym(G,A).$$
\end{lemma}
\begin{proof}
	There is a well known isomorphism between $\H^2(G,A)$ and the group of central extensions of $G$ by~$A$; see for example \cite[Theorem 1.2.4]{neu}. Looking at the definition of this isomorphism, one sees that a cocycle is symmetric if and only if in the associated group extension $0\to A\to B\to G\to 1$ the group $B$ is abelian.
\end{proof}

\begin{lemma}\label{extpgr}
	Let $G$ be a finite abelian group, let $A$ be an abelian group equipped with the trivial $G$-action, and let $p$ be a prime such that $pG=0$. Then there is a (non-canonical) isomorphism
	$$\ext^1(G,A)\simeq\Hom(G,A/pA).$$
\end{lemma}
\begin{proof}
	Using the properties of $\Hom$ and $\ext$ (see for example \cite{wei}), we reduce to the case $G=\Z/p\Z$. From the long exact sequence obtained by applying the $\ext$ functor to
	$$
	0 \longrightarrow \Z \stackrel{p}{\longrightarrow} \Z \longrightarrow \Z/p\Z \longrightarrow 0 
	$$
	and the fact that $\ext^1(\Z,A)=0$, we get an isomorphism $\ext^1(\Z/p\Z,A)\simeq A/pA$. The claim follows using the canonical isomorphism $\Hom(\Z/p\Z,A/pA)\simeq A/pA$.
\end{proof}
\begin{theorem}[K\"unneth formula \cite{jan}]\label{prodcoh}
	Let $G_1$, $G_2$ be profinite abelian groups acting trivially on a discrete abelian group $A$. Then there is a canonical isomorphism
	$$\H^2(G_1\times G_2,A)\simeq \H^2(G_1,A)\oplus \H^2(G_2,A)\oplus \Hom(G_1\otimes G_2,A).$$
\end{theorem}
\begin{definition}\label{embedding}
	Let  $L/K$ be a Galois extension of fields with Galois group~$G$. Consider an extension
	\begin{equation}\label{embprob}
		1\longrightarrow H\longrightarrow \widetilde{G}\stackrel{\pi}{\longrightarrow} G\longrightarrow 1
	\end{equation}
	of profinite groups. A \emph{solution to the embedding problem} relative to the extension $L/K$ and the group extension \eqref{embprob} consists of a Galois extension $M/K$ with $L\subseteq M$ and an isomorphism $\iota\colon\gal(M/K)\isom \widetilde{G}$ such that $\pi\circ\iota$ equals the canonical map $\gal(M/K)\to G$.
\end{definition}

It is not hard to see that the solvability of the embedding problem given by the extension \eqref{embprob} depends only on the equivalence class of the extension.
To conclude this section, we recall a result from \cite{kim} for the case $p=2$ that will be useful later.

\begin{lemma}[{\cite[pp.\ 826--827]{kim}}]\label{trivialclass}
	Let $L/K$ be a Galois extension of fields of characteristic different from~$2$ with Galois group~$G$. Let $1\to\{\pm1\}\to\widetilde{G}\to G\to 1$ be an extension of $G$ by $\{\pm1\}$, and let $c\in \H^2(G,\{\pm 1\})$ be the corresponding cohomology class. Then the embedding problem relative to $L/K$ and the above extension has a solution if and only if $c$ is in the kernel of the natural map
$$\varphi_L\colon \H^2(G,\{\pm1\})\to\H^2(G,L^{\times}).$$
\end{lemma}

\section{Strongly modular elliptic curves}\label{strong_modularity}
The goal of this section is to introduce the main objects of the paper, namely $\Q$-curves and strongly modular $\Q$-curves, and recall some of their basic properties. For more details, see \cite{golaque}, \cite{guique}, \cite{que} and \cite{quer2}.

\begin{definition}
 Let $K$ be a Galois extension of $\Q$ inside $\OQ$. An elliptic curve $E/K$ is called a \emph{$\Q$-curve} if for every $\sigma \in\gal(K/\Q)$ there exists a $\OQ$-isogeny $\mu_{\sigma}\colon \s E\to E$. We say that $E$ is \emph{completely defined} over $K$ if in addition all $\OQ$-isogenies between the $\s E$ are defined over $K$.
\end{definition}

From now on, all our $\Q$-curves will be without CM.

Let $K$ be a Galois extension of $\Q$ and let $E$ be a $\Q$-curve over $K$.  We recall the definition of two cohomology classes
$$
[\lambda]\in\H^1(K/\Q,K^{\times}/(K^{\times})^2),\quad
[\xi_K(E)]\in\H^2(K/\Q,\Q^{\times})
$$
attached to $E$ (the second under the assumption that $E$ is completely defined over~$K$) that encode arithmetic properties of the curve; see also \cite{pyl} or \cite{que}.

Let $G\coloneqq \gal(K/\Q)$. For every $\sigma\in G$ we choose a $\OQ$-isogeny $\mu_{\sigma}\colon \s E\to E$.

By the argument described in \cite[p.~4]{que}, if $E$ is given by an equation of the form $y^2=x^3+Ax+B$ with $A,B\in K$, then for all $\sigma\in G$ and all isogenies $\mu_{\sigma}\colon \s E\to E$, we can write
$$\mu_{\sigma}(x,y)=\biggl(F(x),\frac{1}{\lambda_{\sigma}}yF'(x)\biggr)$$
for some $F(x)\in K(x)$ and $\lambda_{\sigma}\in \OQ^{\times}$ such that $\lambda_{\sigma}^2\in K^{\times}$.
Because $E$ has no CM, for all $\sigma,\tau\in G$ there exists $m(\sigma,\tau)\in \Q^{\times}$ satisfying
\begin{equation}\label{lambda}
 \lambda_{\sigma}\s\lambda_{\tau}=m(\sigma,\tau)\lambda_{\sigma\tau}
\end{equation}
This shows that the map
$$
\begin{aligned}
\lambda\colon G&\to K^{\times}/{(K^{\times})}^2\\
\sigma&\mapsto\lambda_{\sigma}^2
\end{aligned}
$$
is a 1-cocycle for the natural Galois action of $G$ on $K^{\times}/{(K^{\times})}^2$. A calculation shows that the cohomology class $[\lambda]$ of~$\lambda$ depends only on the $K$-isogeny class of $E$.

\begin{definition}\label{minimal_field_def}
  We call the field $K(\lambda_{\sigma}\colon \sigma\in G)$ the \emph{minimal field of complete definition} for $E/K$.
\end{definition}

\begin{proposition}\label{minimal_field}
  Let $L$ be the minimal field of complete definition for $E$. Then:
  \begin{enumerate}[i)]
   \item $L$ is Galois over $\Q$;
   \item $E_L$ is completely defined over $L$;
   \item $L/K$ is a multiquadratic extension;
   \item if $M$ is a Galois extension of $\Q$ containing $K$ and $E_M$ is completely defined over $M$, then $L\subseteq M$.
  \end{enumerate}
\end{proposition}
\begin{proof}
We observe that if $K/F$ is any Galois extension of subfields of $\OQ$ and $\lambda_1,\ldots,\lambda_n\in \OQ$, then the normal closure of $K(\lambda_1,\ldots,\lambda_n)$ over~$F$ equals $K(\s\lambda_i\colon i\in\{1,\ldots,n\},\sigma \in \gal(\overline{\Q}/F))$.
Point i) follows from this observation together with the identity \eqref{lambda}. Points ii), iii) and iv) are clear by construction.
\end{proof}

To construct the other cohomology class attached to $E$, suppose that $E$ is completely defined over $K$, so that every $\mu_{\sigma}$ is itself defined over $K$. The map
  $$
  \begin{aligned}
  \xi_K(E)\colon G\times G &\longrightarrow (\en_K^0(E))^{\times}\simeq\Q^{\times}\\
  (\sigma,\tau) &\longmapsto \mu_{\sigma}\s\mu_{\tau}\mu_{\sigma\tau}^{-1}
  \end{aligned}
  $$
is a 2-cocycle for the trivial action of $G$ on $\Q^{\times}$. With a small abuse of notation, we will often talk about ``the'' 2-cocycle attached to $E$, without mentioning explicitly the system of isogenies giving rise to it. A direct calculation shows that the cohomology class $[\xi_K(E)]$ of $\xi_K(E)$ depends only on the $K$-isogeny class of $E$. When $K=\OQ$, the 2-cocycle attached to $E$ will be denoted simply by $\xi(E)$.
Note that the image of $[\xi_K(E)]$ under the inflation map $\H^2(G,\Q^{\times})\to \H^2(\GQ,\Q^{\times})$ equals $[\xi(E)]$.

The decomposition $\Q^{\times}\simeq \Q^{\times}_+\times\{\pm 1\}$, where $\Q^{\times}_+$ is the multiplicative group of positive rational numbers, yields a decomposition of $\xi_K$ in a 2-cocycle $\xi_K^{\deg}\in \rZ^2(G,\Q^{\times}_+)$ and a 2-cocycle $\xi_K^{\pm}\in \rZ^2(G,\{\pm1\})$. Similarly, the isomorphism
$$\H^2(G,\Q^{\times})\simeq \H^2(G,\Q^{\times}_+)\times \H^2(G,\{\pm 1\})$$
yields a decomposition of $[\xi_K]$ into a \emph{degree component} $[\xi_K^{\deg}]\in \H^2(G,\Q^{\times}_+)$ and a \emph{sign component} $[\xi_K^{\pm}]\in \H^2(G,\{\pm 1\})$.

\begin{definition}
   A $\Q$-curve $E$ over a number field $K$ is \emph{strongly modular} if its $L$-function $L(E/K,s)$ can be written as a product of $L$-functions attached to holomorphic newforms of weight~$2$ for congruence subgroups of the form $\Gamma_1(N)$.
\end{definition}
The newforms in such a product are unique up to ordering \cite[Proposition 3.4]{brufer}. The modularity theorem \cite{bcdt} implies that all elliptic curves over $\Q$ are strongly modular; this is not true for general $\Q$-curves.
\begin{theorem}[{{\cite[Theorem 2.3]{guique2}}}]\label{strongmod}
  Let $E$ be an elliptic curve without CM over a number field $K$. Then $E$ is strongly modular over $K$ if and only if the following three conditions hold:
  \begin{enumerate}[i)]
  \item $K$ is abelian over $\Q$;
  \item $E$ is completely defined over $K$;
  \item the $2$-cocycle $\xi_K$ attached to $E$ is symmetric, i.e.\ $[\xi_K]\in \H^2_\sym(\gal(K/\Q),\Q^{\times})$.
\end{enumerate}
\end{theorem}
\begin{remark}\label{rem1}
	A cohomology class in $\H^2(K/\Q,\Q^{\times})$ is symmetric if and only if both its sign component and its degree component are symmetric. The degree component of $[\xi_K]$ is always symmetric, since for $\sigma,\tau\in \gal(K/\Q)$, the fact that $\xi_K(\sigma,\tau)=\mu_{\sigma}\s\mu_{\tau}\mu_{\sigma\tau}^{-1}$ implies $\xi_K(\sigma,\tau)^2=\deg (\mu_{\sigma})\cdot \deg(\mu_{\tau})\cdot\deg(\mu_{\sigma\tau})^{-1}$, and since $K/\Q$ is an abelian extension, this implies $\xi_K(\sigma,\tau)^2=\xi_K(\tau,\sigma)^2$. Therefore condition~iii) above is equivalent to
	\begin{enumerate}
		\item[iii$'$)] $[\xi_K^{\pm}]\in \H^2_\sym(K/\Q,\{\pm 1\})$.
	\end{enumerate}
\end{remark}

\section{Strongly modular \texorpdfstring{$\Q$}{}-curves up to isomorphism}\label{strong_mod_iso}

Every $\Q$-curve $E$ is geometrically isogenous to a strongly modular one. In fact, it is proved in \cite{gola} that there exists a newform $f$ such that the attached abelian variety $A_f$ admits a non-trivial morphism $(A_f)_{\OQ}\to E$. If $L$ is the splitting field of $f$, it is also proved in \cite{gola} that $A_f$ is isogenous over $L$ to $E'^n$ for some positive integer~$n$, where $E'$ is a $\Q$-curve over $L$. Note that $E$ is geometrically isogenous to $E'$ by the  uniqueness of the decomposition up to isogeny. By \cite[Proposition~2]{gogui}, the restriction of scalars $\res_{L/\Q}(E')$ is isogenous to a product of abelian varieties of the form $A_g$ for some newform $g$. It follows that $E'/L$ is strongly modular, because the $L$-function of $E'/L$ coincides with the $L$-function of $\res_{L/\Q}(E')$; see \cite{mil1}.

Therefore, it is natural to ask how strong modularity behaves with respect to geometric \emph{isomorphisms}, rather than isogenies. It turns out that this is a more rigid property. In fact, we will prove the following theorem.

\begin{theorem}\label{isothm}
	A $\Q$-curve $E/\OQ$ has a strongly modular model over a number field if and only if there exist an abelian number field $K$ and a model of $E$ completely defined over $K$.
\end{theorem}

Compared to Theorem \ref{strongmod}, the above result states that if one considers strong modularity up to geometric isomorphism, then the symmetry of the cocycle attached to the curve is redundant: if a $\Q$-curve is completely defined over an abelian number field, then there exists an appropriate model of the curve whose 2-cocycle is symmetric.

Let us rephrase the theorem above in cohomological terms. Let $E$ be a $\Q$-curve over a Galois number field $K$ with Galois group $G$, and let $\lambda\in \rZ^1(G,K^{\times}/{(K^{\times})}^2)$ be its associated 1-cocycle. A simple computation shows that for all $\gamma\in K^{\times}$, the cocycle attached to the twisted curve $E^{\gamma}$ is given by $\lambda^{\gamma}(\sigma)\coloneqq\lambda(\sigma)\cdot\frac{\s\gamma}{\gamma}$ for every $\sigma\in G$, and is thus cohomologous to $\lambda$. Now let $L$ be an extension of $K$ that is abelian over $\Q$, and let
$$\psi_L\colon \H^1(K/\Q,K^{\times}/{(K^{\times})}^2)\longrightarrow \H^1(L/\Q,L^{\times}/{(L^{\times})}^2)$$
be the composition of the inflation map with the map induced by the natural morphism $K^{\times}/{(K^{\times})}^2\to L^{\times}/{(L^{\times})}^2$. Proposition \ref{minimal_field} implies that $E$ has a model completely defined over $L$ if and only if $\psi_L([\lambda])$ is trivial. It follows that Theorem \ref{isothm} can be stated in the following way: if $E$ is a $\Q$-curve over $K=\Q(j(E))$, then $E$ has a strongly modular model if and only if $K$ is abelian and $[\lambda]$ belongs to the kernel of the map
$$\H^1(K/\Q,K^{\times}/{(K^{\times})}^2)\longrightarrow \H^1(\GQA,{\QA}^{\times}/{({\QA}^{\times})}^2).$$

\begin{example}
	Consider the elliptic curve $E'\colon y^2=x^3+x+1$ over $\Q$. Let $K$ be the non-Galois number field $\Q(\alpha)$, where $\alpha$ is a root of $x^3+x+1$. The base-changed curve $E'_K$ has a non-trivial, rational 2-torsion point, namely $P=(\alpha,0)$. Now let $\varphi$ be the isogeny with kernel $\{O,P\}$ and let $E\coloneqq E'_K/\ker\varphi$. A Weierstrass equation for $E$ is
	$$E\colon y^2=x^3 - (4+15\alpha^2)x + 22+14\alpha.$$
	One can check that
	$$j(E)=\frac{9580464+51659856\alpha+72060192\alpha^2}{961}\notin \O_K,$$
	so $\Q(j(E))=K$ and $E$ has no CM. Moreover, $E$ is a $\Q$-curve: if $L$ is the Galois closure of $K$, then $E_L$ is $L$-isogenous to all its Galois conjugates, since by construction all of them are $L$-isogenous to $E'_L$. Thus, $E$ is a $\Q$-curve that does not have a strongly modular model.
\end{example}

In order to prove Theorem \ref{isothm}, we need two preliminary results.
The first one is proved in \cite{guique} and characterizes the twists of a $\Q$-curve defined over a Galois number field $K$ that are strongly modular.

\begin{lemma}[{{\cite[Lemma 6.1]{guique}}}]\label{quadtw}
	Let $E$ be a $\Q$-curve completely defined over a Galois number field $K$. Let $\gamma\in K^{\times}$, and let $E^{\gamma}$ be the twisted curve.  Let $\xi_K$ and $\xi_{K}^{\gamma}$ be the $2$-cocycles attached to $E$ and $E^\gamma$, respectively. Then $E^{\gamma}$ is completely defined over $K$ if and only if the field $K(\sqrt{\gamma})$ is Galois over $\Q$. In this case, the cohomology classes $[\xi_K]$ and $[\xi_K^{\gamma}]$ in $\H^2(K/\Q,\Q^{\times})$ differ by the cohomology class in $\H^2(K/\Q,\{\pm 1\})$ attached to the exact sequence
		$$1 \longrightarrow \gal(K(\sqrt{\gamma})/K)\simeq \{\pm 1\} \longrightarrow \gal(K(\sqrt{\gamma})/\Q)\longrightarrow \gal(K/\Q)\longrightarrow 1.$$
	In particular, twisting by $\gamma$ affects only the sign component of $[\xi_K]$ and not the degree component.
\end{lemma}

The second preliminary result shows that the 2-torsion of the Brauer group of $\Q$ is generated by cocycles inflated from certain symmetric ones.

\begin{proposition}\label{symmetric_cocycles}
     The following hold:
     \begin{enumerate}[i)]
      \item There is a canonical isomorphism $\displaystyle \bigoplus_{p\;\text{prime}}\{\pm1\} \isom \H^2_\sym(\GQA,\{\pm1\})$.
      \item The inflation map $\infl\colon \H^2_\sym(\GQA,\{\pm1\})\longrightarrow \H^2(\GQ,\{\pm1\})=\Br(\Q)[2]$ is an isomorphism.
     \end{enumerate}
\end{proposition}

\begin{proof}
	To prove i), recall that $\GQA$ is canonically isomorphic to $\widehat{\Z}^{\times}$. By Lemma~\ref{extsym}, we have a canonical isomorphism
	$$
	\varinjlim_n\ext^1((\Z/n\Z)^{\times},\{\pm1\}) \isom \varinjlim_n\H^2_\sym((\Z/n\Z)^{\times},\{\pm1\}) = \H^2_\sym(\widehat{\Z}^{\times},\{\pm1\}).
	$$
	Furthermore, we can rewrite the left-hand side using the canonical isomorphism
	$$
	\bigoplus_{p\;\text{prime}} \varinjlim_r \ext^1((\Z/p^r\Z)^{\times},\{\pm1\}) \isom \varinjlim_n\ext^1((\Z/n\Z)^{\times},\{\pm1\}).
	$$
	It therefore suffices to prove that for every prime number~$p$, there exists a (unique) isomorphism
	$$
	\{\pm1\} \isom \varinjlim_r \ext^1((\Z/p^r\Z)^{\times},\{\pm1\}).
	$$
	For odd $p$, all groups $\ext^1((\Z/p^r\Z)^{\times},\{\pm1\})$, and also their direct limit, are canonically isomorphic to $\ext^1((\Z/p\Z)^{\times},\{\pm1\})$, which has order~$2$ because $(\Z/p\Z)^{\times}$ is cyclic of even order.  For $p=2$, we recall that there are isomorphisms
	$$(\Z/2^r\Z)^{\times}\isom(\Z/4\Z)^{\times}\times\Z/2^{r-2}\Z \quad\text{for all } r\ge2$$
	that are compatible in such a way that we obtain an isomorphism
	$$
	\ext^1((\Z/4\Z)^{\times},\{\pm1\}) \times \varinjlim_r\ext^1(\Z/2^r\Z,\{\pm1\}) \isom \varinjlim_r\ext^1((\Z/2^r\Z)^{\times},\{\pm1\}).$$
	The group $\ext^1((\Z/4\Z)^{\times},\{\pm1\})$ has order~$2$, and it is well known that for all $r\ge1$, the groups $\ext^1(\Z/2^r\Z,\{\pm1\})$ have order~$2$ and the maps $\ext^1(\Z/2^r\Z,\{\pm1\}) \to \ext^1(\Z/2^{r+1}\Z,\{\pm1\})$ are trivial.  This implies the claim.

	To prove ii), consider the exact sequence
	$$1\longrightarrow \gal(\OQ/\QA)\longrightarrow \GQ\longrightarrow \GQA\longrightarrow 1.$$
	The inflation-restriction sequences obtained from the two $\GQ$-modules $\{\pm1\}$ and $\OQ^{\times}$ (with the natural $\GQ$-action) fit into a commutative diagram
	$$
	\begin{xymatrix}{
	   \H^1(\gal(\OQ/\QA),\{\pm 1\})^{\GQA}\ar[r]^-{\trg} \ar[d] & \H^2(\GQA,\{\pm 1\}) \ar[r]^-{\infl}\ar[d]^{\varphi_{\QA}} & \Br(\Q)[2]\ar[d]\\
	    \H^1(\gal(\OQ/\QA),\OQ^{\times})^{\GQA}\ar[r]^-{\trg} & \H^2(\GQA,{\QA}^{\times}) \ar[r]^-{\infl} & \Br(\Q).
	  }\end{xymatrix}$$
	By Hilbert's theorem 90, $\H^1(\gal(\OQ/\QA),\OQ^{\times})$ is trivial, so the map
	$$\infl\colon \H^2(\GQA,{\QA}^{\times}) \longrightarrow \Br(\Q)$$
	is injective. On the other hand, also the map $\Br(\Q)[2]\to \Br(\Q)$ is injective, and therefore we have
	\begin{equation}\label{kernels}
	 \ker\varphi_{\QA}=\ker(\infl\colon \H^2(\GQA,\{\pm 1\})\to \Br(\Q)[2]).
	\end{equation}
	By Lemma \ref{trivialclass}, we have $\H^2_\sym(\GQA,\{\pm1\})\cap \ker\varphi_{\QA}=\{1\}$, since a non-trivial element in this intersection would correspond to a non-trivial extension of $\QA$ that is Galois and abelian over $\Q$. This shows that the map
$$\infl\colon \H^2_\sym(\GQA,\{\pm1\})\to \Br(\Q)[2]$$
	is injective.

	To prove surjectivity, we recall that for every place $v$ of $\Q$, we have a canonical homomorphism $\inv_v\colon\Br(\Q_v)\to \Q/\Z$, and that we have a canonical exact sequence
	$$
	0 \longrightarrow \Br(\Q) \longrightarrow \bigoplus_v\Br(\Q_v) \longrightarrow \Q/\Z \longrightarrow 0,
	$$
	where $v$ runs over all places of $\Q$; the first non-trivial map is the product of the restriction maps $\res_v\colon \Br(\Q)\to\Br(\Q_v)$, and the second one is the sum of the maps $\inv_v$.
	Let $\ell$ be a prime and consider the element $\varepsilon_{\ell}\in \H^2_\sym(\GQA,\{\pm1\})$ that under the isomorphism from i) corresponds to the element $(t_p)_p\in\bigoplus_{p\;\text{prime}}\{\pm1\}$ defined by
	$$t_p=\begin{cases}-1 & \mbox{if } p=\ell\\1 & \mbox{otherwise}.\end{cases}$$
	We will show that $\infl(\varepsilon_{\ell})\in \Br(\Q)[2]$ is ramified precisely at $\ell$ and $\infty$, by looking at the images of $\infl(\varepsilon_\ell)$ under the restriction maps $\res_v$.

	We write $\Q(\zeta_{\ell^\infty})$ for the field obtained by adjoining all roots of unity of $\ell$-power order to~$\Q$.  Let $v$ be a finite place of~$\Q$ different from~$\ell$, and let $\Q_v^\nr$ be the maximal unramified extension of $\Q_v$. Embedding $\Q_v(\zeta_{\ell^{\infty}})$ into $\Q_v^\nr$, we obtain a commutative diagram
	$$
	\begin{xymatrix}{
	\H^2_\sym(\Q(\zeta_{\ell^\infty})/\Q,\{\pm1\}) \ar[rr]^-{\infl}\ar[d]^{\res_v}&& \Br(\Q)\ar[d]^{\res_v}\cr \H^2_\sym(\Q_v(\zeta_{\ell^\infty})/\Q_v,\{\pm1\})\ar[r]^-{\infl} & \Br(\Q_v^\nr/\Q_v)\ar[r]^-\sim \ar[d]^{\mathpzc v_*}& \Br(\Q_v)\cr
	& \H^2(\Q_v^\nr/\Q_v,\Z).
	}\end{xymatrix}
	$$
	Here $\mathpzc v_*$ is induced by the valuation map $\mathpzc v\colon {\Q_v^\nr}^{\times}\to\Z$, and is an isomorphism because the Brauer group of a finite field is trivial \cite[\S X.7 and \S XII.3]{ser3}. Now the composed map $\mathpzc v_*\circ\infl$ vanishes because units have valuation~$0$, so the map $\infl$ in the middle row vanishes as well. Since $\varepsilon_\ell$ is inflated from $\H^2_\sym(\Q(\zeta_{\ell^\infty})/\Q,\{\pm1\})$, it follows that for all finite places $v$ of $\Q$ different from~$\ell$ we have $\res_v(\infl(\varepsilon_{\ell}))=0$ in $\Br(\Q_v)$. Because $\varepsilon_{\ell}\in\H^2_\sym(\GQA,\{\pm1\})$ is non-trivial and the map $\infl\colon \H^2_\sym(\GQA,\{\pm1\})\to \Br(\Q)[2]$ is injective as shown above, it follows that $\infl(\varepsilon_{\ell})\in \Br(\Q)[2]$ is ramified precisely at $\ell$ and~$\infty$, which is what we wanted to show.
\end{proof}

\begin{proof}[Proof of Theorem \ref{isothm}]
	The first implication is clear from Theorem \ref{strongmod}.
	
	Let us now prove the converse. Assume that $E$ has a model $E_0$ completely defined over an abelian number field $K$. Let $\xi_K=\xi_K(E_0)\in \rZ^2(K/\Q,\Q^{\times})$ be the 2-cocycle attached to $E_0$. For every extension $L/K$ that is Galois over $\Q$, we denote by $[\xi_L]$ the inflation of $[\xi_K]$ to $\H^2(L/\Q,\Q^{\times})$. For every Galois extension $M/\Q$, we denote by
	$$\varphi_M\colon \H^2(M/\Q,\{\pm1\})\to \H^2(M/\Q,M^{\times})$$
	the canonical map induced by the inclusion $\{\pm1\}\subseteq M^{\times}$.
	
	We will show that there exists an abelian extension $L/\Q$, containing $K$, such that ${(E_0)}_L$ has a strongly modular twist over $L$. By Theorem \ref{strongmod}, Lemma \ref{quadtw} and Remark \ref{rem1}, this happens if and only if there exist an abelian number field $L$ containing $K$ and an element $\gamma\in L^{\times}$ such that:
	\begin{itemize}
		\item $L(\sqrt{\gamma})$ is Galois over $\Q$;
		\item if $c\in \H^2(L/\Q,\{\pm 1\})$ is the cohomology class attached to the exact sequence
			 $$1\longrightarrow \{\pm 1\}\longrightarrow \gal(L(\sqrt{\gamma})/\Q)\longrightarrow \gal(L/\Q)\longrightarrow 1,$$
			then the class $[\xi_L^{\pm}]\cdot c$ in $\H^2(L/\Q,\{\pm 1\})$ is symmetric.
	\end{itemize}
	By Lemma \ref{trivialclass}, these conditions are satisfied for a given $L$ if and only if there exists $c\in\ker\varphi_L$ such that $[\xi_L^{\pm}]\cdot c$ is symmetric.
	Writing $\H^2(\GQA,\{\pm1\})$ and $\H^2_\sym(\GQA,\{\pm1\})$ as direct limits indexed by finite abelian extensions of~$\Q$ containing $K$, we see that there exist $L$ and $c$ as above if and only if there exists $c'\in\ker\varphi_{\QA}$ such that $[\xi_{\QA}^{\pm}]\cdot c'\in \H^2(\GQA,\{\pm1\})$ is symmetric.

	By \eqref{kernels}, the kernel of $\varphi_{\QA}$ equals the kernel of $\infl\colon \H^2(\GQA,\{\pm1\})\to \Br(\Q)[2]$. Therefore, it remains to prove that there exists $c\in \ker(\infl\colon \H^2(\GQA,\{\pm 1\})\to \Br(\Q)[2])$ such that $[\xi_{\QA}^{\pm}]\cdot c$ is symmetric. This amounts to saying that the image of $[\xi_{\QA}^{\pm}]$ in $\Br(\Q)[2]$ belongs to the image of $\H^2_\sym(\GQA,\{\pm1\})$. By Proposition~\ref{symmetric_cocycles}, this is always the case, and the proof is complete.
\end{proof}

\section{Twists of quadratic \texorpdfstring{$\Q$}{}-curves}\label{section_twist}

Theorem~\ref{isothm} is not effective: it only gives necessary and sufficient conditions for the existence of strongly modular twists, without providing an actual construction of them. In this section, for a $\Q$-curve over a quadratic field $K$, we will make the existence of strongly modular twists over the minimal field of complete definition $L$ effective in terms of the arithmetic of $L$. For certain purposes, such as studying $L$-functions of elliptic curves, it is useful to understand when the curve ``comes from a subfield'' up to isogeny. We will show how to distinguish twists coming from subfields of $L$; this distinction will allow us to characterize all strongly modular twists of $E$ over $K$ (Corollary \ref{corollary_twist}).

\begin{definition}
	Let $E$ be a $\Q$-curve completely defined over a number field $L$. If there exist a subfield $K$ of~$L$ that is Galois over~$\Q$ and a $\Q$-curve $C$ completely defined over $K$ such that $E$ is $L$-isogenous to~$C_L$, we say that $E$ is \emph{inflated from $K$}. If no such $K$ and~$C$ exist, we say that $E$ is \emph{primitive}.
\end{definition}

The reason for the terminology ``inflated'' will be clarified by Proposition \ref{inflation}. We recall the following theorem.

\begin{theorem}[{{\cite[Theorem 8.2]{rib1}}}]\label{descent}
	Let $L/K$ be a Galois extension of number fields, and let $E$ be a $\Q$-curve completely defined over $L$. Then the following are equivalent:
	\begin{enumerate}[i)]
		\item there exist isomorphisms $\mu_{\sigma}\colon \s E\to E$ of elliptic curves up to isogeny over $L$ such that
			  $$\mu_{\sigma}\s\mu_{\tau}=\mu_{\sigma\tau} \qquad \mbox{for all } \sigma,\tau\in \gal (L/K);$$
		\item there exists an elliptic curve $C$ over~$K$ such that $E$ is $L$-isogenous to~$C_L$.
	\end{enumerate}
\end{theorem}

Now let $K$ and~$L$ be Galois number fields with $K\subseteq L$, and consider the inflation map $\infl^K_L\colon \H^2(K/\Q,\Q^{\times})\to \H^2(L/\Q,\Q^{\times})$.
\begin{proposition}\label{inflation}
	Let $E$ be a $\Q$-curve completely defined over $L$ and let $\xi_L(E)$ be its associated $2$-cocycle.
	\begin{enumerate}[i)]
		\item if $E$ is inflated from $K$, then $[\xi_L(E)]\in \im (\infl^K_L)$;
		\item if $\gal(L/K)$ is contained in the center of $\gal(L/\Q)$ and $[\xi_L(E)]\in \im (\infl^K_L)$, then $E$ is inflated from $K$.
	\end{enumerate}
\end{proposition}

\begin{proof}
	First assume that $E$ is inflated from $K$. Let $C$ be a $\Q$-curve completely defined over $K$ such that $C_L$ is isogenous to~$E$. For all $\overline{\sigma}\in\gal(K/\Q)$, we fix a $K$-isogeny $\mu_{\overline{\sigma}}\colon {}^{\overline{\sigma}} C\to C$. Let $\xi_K(C)\in \rZ^2(K/\Q,\Q^{\times})$ be the 2-cocycle attached to~$C$ via the system of isogenies $\{\mu_{\overline{\sigma}}\}_{\overline{\sigma}}$, so that $\xi_K(C)(\overline{\sigma},\overline{\tau})=\mu_{\overline{\sigma}}{}^{\overline{\sigma}}\mu_{\overline{\tau}}\mu_{\overline{\sigma}\overline{\tau}}^{-1}$. Now for every $\sigma\in\gal(L/\Q)$, we set $\mu_{\sigma}=\mu_{\overline{\sigma}}\colon \s C_L\to C_L$, for $\overline{\sigma}$ the class of $\sigma$ in $\gal(K/\Q)$. Then the cocycle $\xi_L(C)$ corresponding to the system of isogenies $\{\mu_{\sigma}\}_{\sigma}$ represents the inflation of $[\xi_K(C)]$; since $C_L$ and $E$ are $L$-isogenous, we have $[\xi_L(C_L)]=[\xi_L(E)]$ and the claim follows.

	To prove ii), note that since $\res^L_K\circ\infl^K_L=0$, we have $[\xi_L(E)]\in \ker(\res^L_K)$. Thus by Theorem \ref{descent} there exists a $\Q$-curve $C$ over $K$ such that $C_L$ is $L$-isogenous to $E$. Choose a system of isogenies $\{\mu_{\sigma}\colon \s C_L\to C_L\}_{\sigma\in\gal(L/\Q)}$ with the following properties:
	\begin{itemize}
		\item $\mu_{\sigma}=1$ whenever $\sigma\in \gal(L/K)$;
		\item $\mu_{\sigma}=\mu_{\tau}$ whenever $\sigma\equiv\tau\bmod\gal(L/K)$.
	\end{itemize}
	Let $\xi_L(C_L)$ be the 2-cocycle attached to $C_L$ via the above system of isogenies. Now suppose that $C$ is not completely defined over $K$. Then there exist $\nu\in \gal(L/\Q)$, $\vartheta\in\gal(L/K)$ such that $\vt\mu_{\nu}=-\mu_{\nu}$; this implies
	\begin{equation}\label{cocycles}
		\xi_L(C_L)(\vartheta,\nu)=-1\mbox{ and }\xi_L(C_L)(\nu,\vartheta)=1.
	\end{equation}
	On the other hand, by hypothesis $[\xi_L(C_L)]=[\xi_L(E)]$ is inflated, so there exists a cocycle $c\in \rZ^2(K/\Q,\Q^{\times})$ such that $[\xi_L(C_L)]=\infl^K_L([c])$. Let $\widetilde{c}\in \rZ^2(L/\Q,\Q^{\times})$ be the cocycle defined by $\widetilde{c}(\sigma,\tau)=c(\overline{\sigma},\overline{\tau})$ for all $\sigma,\tau\in \gal(L/\Q)$, where $\overline{\cdot}$ denotes the equivalence class modulo $\gal(L/K)$. Let $\alpha\colon \gal(L/\Q)\to\Q^{\times}$ be a map such that
	$$\widetilde{c}=\xi_L(C_L)\cdot\delta\alpha.$$
	Note that the cocycle condition for $c$ implies $c(1,\nu)=c(\nu,1)$ for every $\sigma\in\gal(L/\Q)$. Thus, $\widetilde{c}(\nu,\vartheta)=\widetilde{c}(\vartheta,\nu)$ and by \eqref{cocycles}, this yields $\alpha(\nu\vartheta)=-\alpha(\vartheta\nu)$, a contradiction since $\nu\vartheta=\vartheta\nu$.
\end{proof}

From now on, $E$ is a $\Q$-curve without CM over a quadratic field $K=\Q(\sqrt{d})$, for $d\neq 0,1$ a square-free integer. The non-trivial automorphism of $K$ will be denoted by $\nu$. Moreover, we assume that $E/K$ is not strongly modular, i.e.\ that there exists an isogeny $\mu_{\nu}\colon \n E\to E$ of degree $m$ that cannot be defined over $K$.

\begin{lemma}\label{minimal}
	The minimal field of complete definition $L$ of $E$ has Galois group $C_2\times C_2$ over\/~$\Q$.
\end{lemma}

\begin{proof}
	It is clear from the construction of $L$ in the proof of Proposition \ref{minimal_field} that $L$ is a quadratic extension of $K$.

	Suppose that $\gal(L/\Q)\simeq C_4$. Let $\gal(L/\Q)=\{1,\nu,\nu^2,\nu^3\}$, where by a slight abuse of notation, $\nu\in \gal(L/\Q)$ is a lift of $\nu\in \gal(K/\Q)$. We can set $\mu_{\nu^2}=\id$ and $\mu_{\nu^3}=\mu_{\nu}$. Let $\xi_L^{\pm}$ be the sign part of the 2-cocycle attached to $E_L$ via this system of isogenies. Since $\H^2(C_4,\{\pm1\})\simeq C_2\simeq \H^2_\sym(C_4,\{\pm1\})$, the cocycle $\xi_L^{\pm}$ must be symmetric. Then $\xi_L^{\pm}(\nu,\nu^2)=\mu_{\nu}\n\mu_{\nu^2}\mu_{\nu^3}^{-1}=\mu_{\nu}\mu_{\nu}^{-1}=1$. On the other hand, note that $\leftidx{^{\nu^2}\!}\mu_{\nu}$, which is an isogeny $\n E\to E$, cannot coincide with $\mu_{\nu}$, since this would imply that $\mu_{\nu}$ is defined over $K$. But $\mu_{\nu}$ and $\leftidx{^{\nu^2}\!}\mu_{\nu}$ have the same degree, and therefore $\leftidx{^{\nu^2}\!}\mu_{\nu}=-\mu_{\nu}$. Thus $\xi_L^{\pm}(\nu^2,\nu)=\mu_{\nu^2}\leftidx{^{\nu^2}\!}\mu_{\nu}\mu_{\nu^3}^{-1}=-\mu_{\nu}\mu_{\nu}^{-1}=-1$, which contradicts the symmetry of $\xi_L^{\pm}$.	
\end{proof}
		
Let $e\neq 0,1$ be a squarefree integer such that $L=\Q(\sqrt{d},\sqrt{e})$ is the minimal field of complete definition for $E$. From now on, we set
$$K_e\coloneqq\Q(\sqrt{e}),\quad K_{de}\coloneqq\Q(\sqrt{de})\,\, \mbox{ and }\,\, G\coloneqq\gal(L/\Q)=\{1,\nu,\vartheta,\nu\vartheta\},$$
where $\vartheta$ is the generator of $\gal(L/K)$ and by a small abuse of notation the element $\nu\in G$ restricts to the non-trivial automorphism of $K$, which we also call $\nu$.

Let us compute the 2-cocycle $\xi_L\coloneqq\xi_L(E_L)\in\H^2(L/\Q,\Q^\times)$ attached to $E_L$. Let $\mu_{\vartheta}\colon \vt E_L=E_L\to E_L$ be the identity and let $\mu_{\nu\vartheta}=\mu_{\nu}\colon {}^{\nu\vartheta}E_L=\n E_L\to E_L$. Note that $\vt\mu_{\nu}\colon {}^{\nu\vartheta}E_L=\n E_L\to E_L=\vt E_L$ is an isogeny of the same degree as $\mu_{\nu}$; since $E_L$ has no CM, $\vt\mu_{\nu}$ coincides with $\mu_{\nu}$ up to sign. However, if it were $\vt\mu_{\nu}=\mu_{\nu}$, then $\mu_{\nu}$ would be defined over $K$, a contradiction; thus we have $\vt\mu_{\nu}=-\mu_{\nu}$, and hence ${}^{\nu\vartheta}\mu_{\nu}=-\n\mu_{\nu}$. The isogeny $\mu_{\nu}\n\mu_{\nu}$ equals multiplication by an integer $m\in \Z\setminus\{0,1\}$. By an easy computation we end up with the following table for the cocycle~$\xi_L$:

\begin{center}
	\begin{tabular}{| c || c | c | c | c |}
	\hline
	$\xi_L(\cdot,\cdot)$ & $1$ & $\vartheta$ & $\nu$ & $\nu\vartheta$ \\ \hline\hline
	   $1$ & $1$ & $1$ & $1$ & $1$\\ \hline
	    $\vartheta$ & $1$ & $1$ & $-1$ & $-1$\\ \hline
	$\nu$ & $1$ & $1$ & $m$ & $m$\\ \hline
	$\nu\vartheta$ & $1$ & $1$ & $-m$ & $-m$\\  \hline
	\end{tabular}
\end{center}
The sign component $[\xi_L^{\pm}]\in\H^2(L/\Q,\Q^\times)$ is represented by one of the following two non-cohomologous cocycles, depending on the sign of $m$.
	\begin{center}
		\begin{tabular}{| c || c | c | c | c |}
		    \hline
			$\eta_1(\cdot,\cdot)$ & $1$ & $\vartheta$ & $\nu$ & $\nu\vartheta$ \\ \hline\hline
		    $1$ & $1$ & $1$ & $1$ & $1$\\ \hline
		    $\vartheta$ & $1$ & $1$ & $-1$ & $-1$\\ \hline
			$\nu$ & $1$ & $1$ & $1$ & $1$\\ \hline
			$\nu\vartheta$ & $1$ & $1$ & $-1$ & $-1$\\  \hline
		  \end{tabular}\qquad
			\begin{tabular}{| c || c | c | c | c |}
			    \hline
				$\eta_2(\cdot,\cdot)$ & $1$ & $\vartheta$ & $\nu$ & $\nu\vartheta$ \\ \hline\hline
			    $1$ & $1$ & $1$ & $1$ & $1$\\ \hline
			    $\vartheta$ & $1$ & $1$ & $-1$ & $-1$\\ \hline
				$\nu$ & $1$ & $1$ & $-1$ & $-1$\\ \hline
				$\nu\vartheta$ & $1$ & $1$ & $1$ & $1$\\  \hline
		  \end{tabular}
	\end{center}
The table of $\xi_L$ shows that the curve $E_L$ is not strongly modular over $L$, because of Theorem \ref{strongmod}. The question we want to address is: which quadratic twists of $E_L$ are strongly modular? A first answer is provided by the following lemma.

\begin{lemma}\label{stronglymodulartwists}
  Let $\gamma\in L^{\times}$. Then the twisted curve $E_L^{\gamma}$ is strongly modular over $L$ if and only if $L(\sqrt{\gamma})$ is Galois and non-abelian over\/~$\Q$.
\end{lemma}

\begin{proof}
  Let $\xi_L^{\gamma}$ be the cocycle attached to $E_L^{\gamma}$. By Theorem \ref{strongmod}, Lemma \ref{quadtw} and Remark \ref{rem1}, the curve $E_L^{\gamma}$ is strongly modular if and only if $L(\sqrt{\gamma})$ is Galois over~$\Q$ and $[\xi_L^{\gamma,\pm}]\in \H^2_\sym(L/\Q,\{\pm 1\})$.
  The group $\widetilde{G}\coloneqq\gal(L(\sqrt{\gamma})/\Q)$ is abelian if and only if the 2-cocycle attached to the exact sequence $1\to {\pm1} \to \widetilde{G}\to G\to 1$ is symmetric. Therefore when $\widetilde{G}$ is abelian, by Lemma \ref{quadtw} the symmetry of the cocycle~$\xi_L$ attached to $E_L$ does not change under twisting by $\gamma$, and this shows that $E_L^{\gamma}$ cannot be strongly modular. On the other hand, by Lemmas \ref{extsym} and~\ref{extpgr} we have $\H^2_\sym(G,\{\pm1\})\simeq C_2\times C_2$, while by Theorem \ref{prodcoh} we have $\H^2(G,\{\pm1\})\simeq C_2^3$. This shows that $\H^2(G,\{\pm1\})/\H^2_\sym(G,\{\pm1\})\simeq C_2$, which means that the product of two non-symmetric classes in $\H^2(G,\{\pm1\})$ is symmetric. Whenever $\widetilde{G}$ is non-abelian, we therefore have $[\xi_L^{\gamma}]\in \H^2_\sym(G,\{\pm1\})$.
\end{proof}

\begin{remark}\label{C2_cohomology}
  Let us describe the structure of $\H^2(L/\Q,\{\pm 1\})$ more in detail. Recall that elements of this group correspond to equivalence classes of central extensions of the form $1\to \{\pm 1\}\to \widetilde{G}\to G\to 1$. There are four symmetric cohomology classes and four non-symmetric ones. The symmetric classes correspond to extensions with $\widetilde{G}\simeq C_2\times C_2\times C_2$ or $\widetilde{G}\simeq C_4\times C_2$. The non-symmetric classes correspond to extensions with $\widetilde{G}\simeq D_4$, the dihedral group of order $8$, or $\widetilde{G}\simeq H_8$, the group of quaternions. All extensions with $\widetilde{G}\simeq H_8$ are equivalent to each other, and the corresponding cohomology class is represented by the following cocycle:
  \begin{center}
    \begin{tabular}{| c || c | c | c | c |}
	   \hline
	$h_0(\cdot,\cdot)$ & $1$ & $\vartheta$ & $\nu$ & $\nu\vartheta$ \\ \hline\hline
	$1$ & $1$ & $1$ & $1$ & $1$\\ \hline
	$\vartheta$ & $1$ & $-1$ & $-1$ & $1$\\ \hline
	$\nu$ & $1$ & $1$ & $-1$ & $-1$\\ \hline
	$\nu\vartheta$ & $1$ & $-1$ & $1$ & $-1$\\  \hline
    \end{tabular}
  \end{center}
  On the other hand, there are three non-equivalent extensions with $\widetilde{G}\simeq D_4$. These are uniquely determined by the image in $G$ of the cyclic subgroup of order 4 in $D_4$. If $\gamma\in L^{\times}$ is such that $\widetilde{G}=\gal(L(\sqrt{\gamma})/\Q)\simeq D_4$, $\sigma \in \widetilde{G}$ is an element of order 4 and $\overline{\sigma}$ is its image in $G$, then $L^{\overline{\sigma}}$ is the unique subextension such that $\gal(L(\sqrt{\gamma})/L^{\overline{\sigma}})\simeq C_4$. The following three cocycles represent these classes. The $C_4$-subextension is $L(\sqrt{\gamma})/K,L(\sqrt{\gamma})/K_e,L(\sqrt{\gamma})/K_{de}$, respectively.
	\begin{center}
		\begin{tabular}{| c || c | c | c | c |}
	    \hline
		$h_d(\cdot,\cdot)$ & $1$ & $\vartheta$ & $\nu$ & $\nu\vartheta$ \\ \hline\hline
	    $1$ & $1$ & $1$ & $1$ & $1$\\ \hline
	    $\vartheta$ & $1$ & $-1$ & $1$ & $-1$\\ \hline
		$\nu$ & $1$ & $-1$ & $1$ & $-1$\\ \hline
		$\nu\vartheta$ & $1$ & $1$ & $1$ & $1$\\  \hline
	  \end{tabular}\qquad\qquad
		\begin{tabular}{| c || c | c | c | c |}
	    \hline
		$h_e(\cdot,\cdot)$ & $1$ & $\vartheta$ & $\nu$ & $\nu\vartheta$ \\ \hline\hline
		    $1$ & $1$ & $1$ & $1$ & $1$\\ \hline
		    $\vartheta$ & $1$ & $1$ & $1$ & $1$\\ \hline
			$\nu$ & $1$ & $-1$ & $-1$ & $1$\\ \hline
			$\nu\vartheta$ & $1$ & $-1$ & $-1$ & $1$\\  \hline
		  \end{tabular}
	\end{center}

	\begin{center}
		\begin{tabular}{| c || c | c | c | c |}
		    \hline
				$h_{de}(\cdot,\cdot)$ & $1$ & $\vartheta$ & $\nu$ & $\nu\vartheta$ \\ \hline\hline
			    $1$ & $1$ & $1$ & $1$ & $1$\\ \hline
			    $\vartheta$ & $1$ & $1$ & $1$ & $1$\\ \hline
				$\nu$ & $1$ & $-1$ & $1$ & $-1$\\ \hline
				$\nu\vartheta$ & $1$ & $-1$ & $1$ & $-1$\\  \hline
		  \end{tabular}
	\end{center}
\end{remark}

\subsection{Distinguishing inflated and primitive twists}

  Let $\gamma\in L^{\times}$ be such that $E_L^{\gamma}$ is strongly modular. By Lemma \ref{stronglymodulartwists}, $L(\sqrt{\gamma})$ is a non-abelian Galois extension of $\Q$, and $\widetilde{G}\coloneqq \gal(L(\sqrt{\gamma})/\Q)$ falls in precisely one of the following cases:
  \begin{enumerate}[A.]
    \item $\widetilde{G}\simeq H_8$;
    \item $\widetilde{G}\simeq D_4$ and the unique $C_4$-subextension is $L(\sqrt{\gamma})/K$;
    \item $\widetilde{G}\simeq D_4$ and the unique $C_4$-subextension is $L(\sqrt{\gamma})/K_e$;
    \item $\widetilde{G}\simeq D_4$ and the unique $C_4$-subextension is $L(\sqrt{\gamma})/K_{de}$.
  \end{enumerate}
Recall that the sign component $\xi_L^{\pm}$ of the 2-cocycle attached to $E_L$ equals one of the two classes $[\eta_1]$ and $[\eta_2]$ described below Lemma \ref{minimal}.
\begin{lemma}\label{inducedQcurves}
    If $|m|\in (\Q^{\times})^2$, the curve $E_L^{\gamma}$ is inflated from the subfield $F\subseteq L$, according to the following table:
      \begin{center}
		\begin{tabular}{| c || c | c | c | c |}
		    \hline
				\backslashbox{$[\xi_L^{\pm}]$}{case} & $A.$ & $B.$ & $C.$ & $D.$ \\ \hline\hline
			    $[\eta_1]$ & $K_{de}$ & $K_e$ & $K$ & $\Q$\\ \hline
			    $[\eta_2]$ & $K_e$ & $K_{de}$ & $\Q$ & $K$\\ \hline
		  \end{tabular}
	\end{center}
    If $|m|\notin (\Q^{\times})^2$, the curve $E_L^{\gamma}$ is inflated from $K$ if and only if C.\ or D.\ holds, and it is primitive if and only if A.\ or B.\ holds.
\end{lemma}

\begin{proof}
  By Proposition \ref{inflation} and the fact that the inflation map preserves the degree and the sign components and the subgroup of symmetric classes, $E_L^{\gamma}$ is inflated from a Galois subfield $F\subseteq L$ if and only if $[\xi_L^{\gamma,\pm}]$ is the inflation of a cohomology class in $\H^2_\sym(F/\Q,\{\pm 1\})$ and $[\xi_L^{\gamma,\deg}]$ is the inflation of a cohomology class in $\H^2_\sym(F/\Q,\Q^{\times}_+)$.

  By Lemma \ref{quadtw}, $[\xi_L^{\gamma}]$ is the product of $[\xi_L]$ with the class $[t]$ corresponding to the exact sequence
  $$
  1 \longrightarrow \gal(L(\sqrt{\gamma})/L) \longrightarrow \gal(L(\sqrt{\gamma})/\Q) \longrightarrow \gal(L/\Q) \longrightarrow 1.
  $$
  The elements $h_0$, $h_d$, $h_e$ and $h_{de}$ described in Remark \ref{C2_cohomology} represent cases A., B., C.\ and D.\, respectively.
  
  The degree component of $\xi_L$ coincides with that of $\xi_L^{\gamma}$ and is represented by the following cocycle:
	\begin{center}
		\begin{tabular}{| c || c | c | c | c |}
		    \hline
			$\xi_L^{\deg}(\cdot,\cdot)$ & $1$ & $\vartheta$ & $\nu$ & $\nu\vartheta$ \\ \hline\hline
		    $1$ & $1$ & $1$ & $1$ & $1$\\ \hline
		    $\vartheta$ & $1$ & $1$ & $1$ & $1$\\ \hline
			$\nu$ & $1$ & $1$ & $|m|$ & $|m|$\\ \hline
			$\nu\vartheta$ & $1$ & $1$ & $|m|$ & $|m|$\\  \hline
		  \end{tabular}
	\end{center}
  
  If $|m|\in (\Q^{\times})^2$, then $[\xi_L^{\deg}]$ is trivial. Thus $[\xi_L^{\gamma}]$ is inflated if and only if $[\xi_L^{\gamma,\pm}]$ is inflated. The non-trivial symmetric classes in $\H^2(L/\Q,\{\pm1\})$ are represented by the following cocycles $b_d$, $b_e$, $b_{de}$:
	\begin{center}
		\begin{tabular}{| c || c | c | c | c |}
	    \hline
		$b_d(\cdot,\cdot)$ & $1$ & $\vartheta$ & $\nu$ & $\nu\vartheta$ \\ \hline\hline
	    $1$ & $1$ & $1$ & $1$ & $1$\\ \hline
	    $\vartheta$ & $1$ & $1$ & $1$ & $1$\\ \hline
		$\nu$ & $1$ & $1$ & $-1$ & $-1$\\ \hline
		$\nu\vartheta$ & $1$ & $1$ & $-1$ & $-1$\\  \hline
	  \end{tabular}\qquad\qquad
		\begin{tabular}{| c || c | c | c | c |}
	    \hline
		$b_e(\cdot,\cdot)$ & $1$ & $\vartheta$ & $\nu$ & $\nu\vartheta$ \\ \hline\hline
		    $1$ & $1$ & $1$ & $1$ & $1$\\ \hline
		    $\vartheta$ & $1$ & $-1$ & $1$ & $-1$\\ \hline
			$\nu$ & $1$ & $1$ & $1$ & $1$\\ \hline
			$\nu\vartheta$ & $1$ & $-1$ & $1$ & $-1$\\  \hline
		  \end{tabular}
	\end{center}

	\begin{center}
		\begin{tabular}{| c || c | c | c | c |}
		    \hline
				$b_{de}(\cdot,\cdot)$ & $1$ & $\vartheta$ & $\nu$ & $\nu\vartheta$ \\ \hline\hline
			    $1$ & $1$ & $1$ & $1$ & $1$\\ \hline
			    $\vartheta$ & $1$ & $-1$ & $-1$ & $1$\\ \hline
				$\nu$ & $1$ & $-1$ & $-1$ & $1$\\ \hline
				$\nu\vartheta$ & $1$ & $1$ & $1$ & $1$\\  \hline
		  \end{tabular}
	\end{center}
	It is immediately clear that $[b_d]$ (resp.\ $[b_e]$, $[b_{de}]$) is the inflation of the unique non-trivial element in $\H^2(K/\Q,\{\pm1\})$ (resp.\ $\H^2(K_e/\Q,\{\pm1\})$, $\H^2(K_{de}/\Q,\{\pm1\})$). Thus our claim is equivalent to showing that the multiplication table of $[\eta_1],[\eta_2]$ by $[h_0],[h_d],[h_e],[h_{de}]$ is the following:
  \begin{center}
    \begin{tabular}{| c || c | c | c | c |}
	    \hline
	    $\cdot$ & $[h_0]$ & $[h_d]$ & $[h_e]$ & $[h_{de}]$ \\ \hline\hline
	   $[\eta_1]$ & $[b_{de}]$ & $[b_e]$ & $[b_d]$ & $1$\\ \hline
	   $[\eta_2]$ & $[b_e]$ & $[b_{de}]$ & $1$ & $[b_d]$\\ \hline
    \end{tabular}
  \end{center}
  and it is easy to check that this is the case.
	
  Assume now that $|m|\notin {(\Q^{\times})}^2$. The class $[\xi_L^{\deg}]$ is the inflation from $\H^2(K/\Q,\Q^{\times}_+)$ of the class $[c]$, where $c(1,1)=c(1,\nu)=c(\nu,1)=1$ and $c(\nu,\nu)=|m|$, while it is not the inflation of a class lying in $\H^2(F/\Q,\Q^{\times}_+)$ for any $F\in \{\Q,K_e,K_{de}\}$. Thus $E_L^{\gamma}$ is inflated if and only if it is inflated from $K$.

  Therefore it is enough to check when $[\xi_L^{\gamma,\pm}]$ coincides with $[b_K]$. Since $[b_K]=[\eta_1]\cdot[\eta_2]$, the class $[\xi_L^{\gamma}]$ is inflated if and only if the class $[t]$ equals either $[\eta_1]$ or $[\eta_2]$. It is immediate to see that $[\eta_1]=[h_{de}]$ and $[\eta_2]=[h_e]$, and the proof is complete.
\end{proof}

The next step is to give necessary and sufficient conditions for the existence of primitive or inflated twists of $E_L$. Recall that $L=\Q(\sqrt{d},\sqrt{e})$ is a $C_2\times C_2$-extension of $\Q$ and $G=\gal(L/\Q)=\langle\nu,\vartheta\rangle$ where $\n\sqrt{d}=-\sqrt{d}$, $\vt\sqrt{e}=-\sqrt{e}$. For $a,b\in \Q$, we will denote by $(a,b)$ the quaternion algebra over $\Q$ with basis $\{1,i,j,ij\}$ such that $i^2=a$, $j^2=b$, $ij=-ji$. Recall that the \emph{reduced discriminant} of a quaternion algebra $B$ over $\Q$ is the product of the finite primes of $\Q$ where $B$ ramifies. A quaternion algebra is trivial in $\Br(\Q)$ if and only if it has reduced discriminant $1$.

\begin{theorem}[{{\cite[Theorems 4 and 5]{kim}}}]\label{embsol}
  The following hold:
  \begin{enumerate}[i)]
    \item Let $H_8=\{\pm 1,\pm i,\pm j,\pm k\}$ be the group of quaternions. The embedding problem (cf.\ Definition \ref{embedding}) relative to $L/\Q$ and the group extension
    $$1\longrightarrow C_2\longrightarrow H_8\stackrel{\pi}{\longrightarrow} G\longrightarrow 1$$
    is solvable if and only if $(d,de)(e,de)(d,e)=1$ in $\Br(\Q)$, if and only if there exist $v_1,v_2,v_3,w_1,w_2,w_3\in \Q$ such that
    $$\begin{cases}
    d=v_1^2+v_2^2+v_3^2 & \\
    e=w_1^2+w_2^2+w_3^2 & \\
    v_1w_1+v_2w_2+v_3w_3=0.
    \end{cases}$$
    In this case, setting $t=1+\frac{v_1}{\sqrt{d}}+\frac{w_3}{\sqrt{e}}+\frac{v_1w_3-v_3w_1}{\sqrt{de}}$, the extensions solving the problem are exactly the ones of the form $L(\sqrt{qt})$, for $q\in \Q^{\times}$.
    \item Let $D_4=\langle \sigma,\tau\colon \sigma^4=\tau^2=1,\,\, \sigma\tau=\tau\sigma^3\rangle$ be the dihedral group of order $8$. The embedding problem relative to $L/\Q$ and the group extension
    $$1\longrightarrow C_2\longrightarrow D_4\stackrel{\pi}{\longrightarrow} G\longrightarrow 1$$
    where $\pi(\sigma)=\vartheta$ and $\pi(\tau)=\nu$ is solvable if and only if $(-d,e)=1$ in $\Br(\Q)$.

    In this case, if $x,y\in \Q$ are such that $d=ey^2-x^2$, the extensions solving this problem are exactly the ones of the form $L(\sqrt{q(ey+x\sqrt{e})})$ for $q\in \Q^{\times}$.
  \end{enumerate}
\end{theorem}

Recall that the class $[\xi_L^{\gamma,\pm}]$ equals one of the classes $\eta_1,\eta_2$ described below Lemma~\ref{minimal}.

\begin{theorem}\label{primQcurves}
  There exists $\gamma\in L^{\times}$ such that $E^{\gamma}$ is strongly modular if and only if at least one of the following conditions is satisfied:
  \begin{enumerate}
    \item[A.] the quaternion algebra $(-d,-e)$ has reduced discriminant 2;
    \item[B.] the quaternion algebra $(-d,e)$ is trivial in $\Br(\Q)$;
    \item[C.] the quaternion algebra $(d,-e)$ is trivial in $\Br(\Q)$;
    \item[D.] the quaternion algebra $(d,-de)$ is trivial in $\Br(\Q)$.
  \end{enumerate}
  In particular, if $m\in {(\Q)^{\times}}^2$ then there exists $\gamma$ such that $E^{\gamma}$ is inflated from the subfield $F$, according to the following table:
  \begin{center}
    \begin{tabular}{| c || c | c | c | c |}
    \hline
      \backslashbox{$[\xi_L^{\pm}]$}{case} & $A.$ & $B.$ & $C.$ & $D.$ \\ \hline\hline
      $[\eta_1]$ & $K_{de}$ & $K_e$ & $K$ & $\Q$\\ \hline
      $[\eta_2]$ & $K_e$ & $K_{de}$ & $\Q$ & $K$\\ \hline
      \end{tabular}
  \end{center}
  If $m\notin {(\Q)^{\times}}^2$, then there exists $\gamma$ such that $E^{\gamma}$ is primitive if and only if A. or B. holds, while there exists $\gamma$ such that $E^{\gamma}$ is inflated from $K$ if and only if C. or D. holds.
\end{theorem}

\begin{proof}
  By Lemma \ref{stronglymodulartwists}, $E^{\gamma}$ is strongly modular if and only if $L(\sqrt{\gamma})/\Q$ is a non-abelian Galois extension. Thus, there exists $\gamma$ such that $E^{\gamma}$ is strongly modular if and only if the embedding problem 
   \begin{equation}\label{nonabelian}
     1 \longrightarrow \gal(L(\sqrt{\gamma})/L) \longrightarrow \widetilde{G} \longrightarrow \gal(L/\Q) \longrightarrow 1
  \end{equation}
  has a solution with $\widetilde{G}$ non-abelian.
	
  When $\widetilde{G}\simeq H_8$, by Theorem \ref{embsol} and the discussion at \cite[p.~239]{kim}, the embedding problem \eqref{nonabelian} is solvable if and only if the quadratic forms $S_{d,e}=\frac{1}{de}X^2+dY^2+eZ^2$ and $T=X^2+Y^2+Z^2$ are equivalent over $\Q$. This implies immediately that $d,e>0$ because $S_{d,e}$ and $T$ must have the same signature. Since the rank and the discriminant obviously coincide, it only remains to check that the Hasse-Witt invariants coincide. If $p$ is a prime, the Hasse-Witt invariant of $T$ at $p$ is $1$, while the Hasse-Witt invariant of $S_{d,e}$ at $p$ is
  $$
  \begin{aligned}
  (de,d)_p(de,e)_p(d,e)_p&=(de,d)_p(de,e)_p(-d,-e)_p(d,-1)_p(-1,e)_p(-1,-1)_p=\\
  &=(de,-de)_p(-d,-e)_p(-1,-1)_p=\\
  &=(-d,-e)_p(-1,-1)_p,
  \end{aligned}
  $$
  where we used bilinearity of the Hilbert symbol and the fact that $(a,-a)_p=1$ for every $a\in \Q^{\times}$ and every prime $p$. Since $(-1,-1)$ ramifies precisely at 2 and $\infty$, we see that this instance of the embedding problem is solvable if and only if A.\ holds.

  When $\widetilde{G}\simeq D_4$, point ii) of Theorem \ref{embsol} shows that the embedding problem \eqref{nonabelian} is solvable if and only if B., C.\ or D.\ holds.

  The other claims follow immediately from Lemma \ref{inducedQcurves}.
\end{proof}

\begin{corollary}\label{corollary_twist}
  The curve $E$ has a strongly modular quadratic twist over $K$ if and only if the curve $E_L$ has a strongly modular twist that is inflated from $K$.
\end{corollary}

\begin{proof}
  First recall that, by Theorem \ref{strongmod}, a $\Q$-curve over a quadratic field~$K$ is strongly modular if and only if it is completely defined over~$K$.

  Let $\gamma\in K^{\times}$ be such that $E^{\gamma}$ is strongly modular over $K$. Then the base-changed curve $(E^{\gamma})_L$ is strongly modular over $L$ since its attached cocycle is the inflation of a symmetric one, and is therefore symmetric. On the other hand, $(E^{\gamma})_L$ is isomorphic to $(E_L)^{\gamma}$, and hence $E_L$ has a strongly modular twist inflated from~$K$.

  Conversely, if $E_L$ has a strongly modular twist $(E_L)^{\gamma}$ inflated from $K$, then by Theorem \ref{primQcurves} at least one of $(d,-e)$ or $(d,-de)$ is trivial. Theorem \ref{embsol} shows that we can choose $\gamma\in K^{\times}$: it is enough to use point ii) of the theorem, replacing the map~$\pi$ by the one given by $\pi(\tau)=\vartheta$ and $\pi(\sigma)=\nu$ or $\pi(\sigma)=\nu\vartheta$. Then $d$ plays the role of~$e$ in the notation of the theorem, and it is clear that $\gamma\in K^{\times}$. Therefore $(E_L)^{\gamma}$ is $L$-isomorphic to $(E^{\gamma})_L$. Now $E^{\gamma}$ is completely defined over $K$, otherwise $E^{\gamma}_L$ would not be strongly modular since its attached cocycle would not be symmetric (cf.\ the discussion below Lemma \ref{minimal}).
\end{proof}

\subsection{Examples}

  We will now give examples of $\Q$-curves with different behaviours with respect to the existence of primitive and inflated strongly modular quadratic twists.

\subsubsection*{Example 1}

  The following example is borrowed from \cite{pyl}. Let $E$ be the following elliptic curve without CM over $K=\Q(\sqrt{-3})$:
  $$E\colon y^2=x^3+2x^2+bx,$$
  where $b\in \O_K$ is any element of trace $1$. There is an isogeny $\mu_{\nu}\colon \n E\to E$ such that $\mu_{\nu}\n\mu_{\nu}=-2$. The minimal field of definition of $E$ is $L=\Q(\sqrt{-3},\sqrt{-2})$. Since $(3,-2)$ is trivial in $\Br(\Q)$, by Theorem \ref{primQcurves} there are quadratic extensions of $L$ of type $D_4$ over $\Q$ with $C_4$-subextension $L(\sqrt{\gamma})/K$. Since $\alpha=1+\sqrt{-2}\in \Q(\sqrt{e})$ has norm $3=-d$, by Theorem \ref{embsol}, the set of all these extensions is $\left\{L\biggl(\sqrt{r+r/\sqrt{-2}}\biggr)\colon r\in \Q^{\times}\right\}$. The one found in \cite{pyl} corresponds to $r=2$. Let $\gamma=2-\sqrt{-2}$. An integral model for $E_L^{\gamma}$ is
  $$E_L^{\gamma}\colon y^2=x^3+(4-2\sqrt{-2})x^2+b(2-4\sqrt{-2})x.$$

  By Theorem \ref{primQcurves} there are no quadratic extensions of $L$ that are of type $H_8$ over~$\Q$, nor quadratic extension of type $D_4$ with $C_4$-subextension $\Q(\sqrt{-2})$ or $\Q(\sqrt{6})$. Thus, all strongly modular quadratic twists of $E$ are primitive over $L$. Note also that \cite[Proposition 6.2]{pyl}, which asserts that there are no quadratic twists of $E$ that are completely defined over $K$, follows immediately from Corollary \ref{corollary_twist}.
 
  To construct other examples, consider the following family of $\Q$-curves given in~\cite{que}:
  $$E_a\colon y^2=x^3-3\sqrt{a}(4+5\sqrt{a})x+2\sqrt{a}(2+14\sqrt{a}+11a),$$
  where $a\in \Z$ is not a square. Then $E_a$ is defined over $K_a\coloneqq\Q(\sqrt{a})$, but its minimal field of complete definition is $L_a\coloneqq K_a(\sqrt{3})$.

\subsubsection*{Example 2}

  Consider the curve $(E_6)_{L_6}$. With the notation of Theorem \ref{primQcurves}, we have $d=6$ and $e=3$. Since $j(E_6)=\frac{27625536+10768896\sqrt{6}}{125}$, it follows that $E_6$ has no CM.

  The quaternion algebra $(-6,-3)=(-2,-3)$ has reduced discriminant 2, and therefore by Theorem \ref{primQcurves} it follows that $L_6$ has a quadratic extension of type~$H_8$ over~$\Q$. Following the notation of Theorem \ref{embsol}, we can pick $v_1=2$, $v_2=v_3=1$, $w_1$ and $w_2=w_3=-1$. Thus $t=1+\frac{2}{\sqrt{6}}-\frac{1}{\sqrt{3}}-\frac{1}{\sqrt{2}}$ and all extensions of type $H_8$ of $L_6$ are of the form $L_6(\sqrt{qt})$ with $q\in \Q^{\times}$. For example, letting $q=1$ and $\gamma=t$, an integral model for $(E_6)_{L_6}^{\gamma}$ is
  $$(E_6)_{L_6}^{\gamma}\colon y^2 = x^3+Ax+B,$$ 
  where
  $$
  \begin{aligned}
  A&=4080384\alpha^3-13616640\alpha^2-412416\alpha+1375488,\\
  B&=-25868537856\alpha^3+82215567360\alpha^2+2613252096\alpha-8305459200\\
  \end{aligned}
  $$
  and $\alpha=\sqrt{2}+\sqrt{3}$. This is a primitive strongly modular curve over $L_6$.

  Since $(-6,3)$ and $(6,-3)$ both have reduced discriminant $6$, there are no extensions of $L_6$ that are of type $D_4$ over $\Q$ with $C_4$-subextension $L_6/\Q(\sqrt{6})$ or $L_6/\Q(\sqrt{3})$. On the other hand, $(6,-18)$ is trivial in $\Br(\Q)$, so by Corollary \ref{corollary_twist} there exist strongly modular twists of $E_6$. To find them, note that by Theorem~\ref{embsol} it is enough to find $x,y\in \Q$ with $6y^2-x^2=18$. As $x=6$, $y=3$ solve this equation, letting $t=18+6\sqrt{6}$ we get that all extensions of $L_6$ of type $D_4$ over $\Q$ are of the form $L_6(\sqrt{qt})$ with $q\in \Q^{\times}$. Let us choose for example $q=1/6$ and $\gamma'=qt$. Then an integral model for $E_6^{(\gamma')}$ is
  $$E_6^{(\gamma')}\colon y^2 = x^3 - (28512+11520 \sqrt{6}) x + 2594304+ 1059840 \sqrt{6}.$$

\subsubsection*{Example 3}

  The curve $E_7$ is not strongly modular, but since $(7,-3)$ is trivial, by Corollary \ref{corollary_twist} it has strongly modular quadratic twists. For example setting $\gamma=7+2\sqrt{7}$, the curve $E^{\gamma}$ is strongly modular over $K_7$. An integral model is
  $$E_7^{\gamma}\colon y^2 = x^{3} - (166992+61824 \sqrt{7}) x + 36452864+ 13804672 \sqrt{7}.$$
  Since $(-7,-3)$ has reduced discriminant $3$ and $(-7,3)$ has reduced discriminant~$21$, Theorem \ref{primQcurves} shows that $(E_7)_{L_7}$ has no primitive strongly modular twists.

\subsubsection*{Example 4}

  Finally, the curve $(E_5)_{L_5}$ has no strongly modular twists at all, since $(3,-5)$ has reduced discriminant~$10$, $(-3,-5)$ has reduced discriminant $5$ and both $(5,-3)$ and $(5,-15)$ have reduced discriminant~$15$.

\bibliographystyle{plain}
\bibliography{biblio}

\begin{thebibliography}{10}

\bibitem{bcdt}
C.~Breuil, B.~Conrad, F.~Diamond, and R.~Taylor.
\newblock On the modularity of elliptic curves over $\mathbb{Q}$: wild $3$-adic
  exercises.
\newblock {\em J. Amer. Math. Soc.}, 14(4):843--939 (electronic), 2001.

\bibitem{brufer}
P.~J. Bruin and A.~Ferraguti.
\newblock On {$L$}-functions of quadratic {$\mathbb Q$}-curves.
\newblock {\em Math. Comp. (to appear)}, 2017.

\bibitem{gola}
J.~Gonz\'{a}lez and J.-C. Lario.
\newblock $\mathbb{Q}$-curves and their {M}anin ideals.
\newblock {\em Amer. J. Math.}, 123(3):475--503, 2001.

\bibitem{golaque}
Josep Gonz\'alez, Joan-Carles Lario, and Jordi Quer.
\newblock Arithmetic of {$\mathbb Q$}-curves.
\newblock In {\em Modular curves and abelian varieties}, volume 224 of {\em
  Progr. Math.}, pages 125--139. Birkh\"auser, Basel, 2004.

\bibitem{gogui}
E.~Gonz\'{a}lez-Jim\'{e}nez and X.~Guitart.
\newblock On the modularity level of modular abelian varieties over number
  fields.
\newblock {\em J. Number Theory}, 130:1560--1570, 2010.

\bibitem{guique2}
X.~Guitart and J.~Quer.
\newblock Remarks on strongly modular {J}acobian surfaces.
\newblock {\em J. Th\'eor. Nombres Bordeaux}, 23(1):171--182, 2011.

\bibitem{guique}
X.~Guitart and J.~Quer.
\newblock Modular abelian varieties over number fields.
\newblock {\em Canad. J. Math.}, 66(1):170--196, 2014.

\bibitem{jan}
U.~Jannsen.
\newblock The splitting of the {H}ochschild-{S}erre spectral sequence for a
  product of groups.
\newblock {\em Canad. Math. Bull.}, 33:181--183, 1990.

\bibitem{kwin}
C.~Khare and J.-P. Wintenberger.
\newblock Serre's modularity conjecture. {I}.
\newblock {\em Invent. Math.}, 178(3):485--504, 2009.

\bibitem{kim}
I.~Kiming.
\newblock Explicit classification of some $2$-extensions of a field of
  characteristic different from $2$.
\newblock {\em Canad. J. Math.}, XLII(5):825--855, 1990.

\bibitem{mil1}
J.~S. Milne.
\newblock On the arithmetic of abelian varieties.
\newblock {\em Invent. Math.}, 17:177--190, 1972.

\bibitem{neu}
J.~Neukirch, A.~Schmidt, and K.~Wingberg.
\newblock {\em Cohomology of Number Fields}.
\newblock Springer, 2000.

\bibitem{pyl}
E.~E. Pyle.
\newblock Abelian varieties over {$\Q$} with large endomorphism algebras and
  their simple components over {$\overline{\Q}$}.
\newblock In {\em Modular curves and abelian varieties}, volume 224 of {\em
  Progr. Math.}, pages 189--239. Birkh\"auser, Basel, 2004.

\bibitem{que}
J.~Quer.
\newblock $\mathbb{Q}$-curves and abelian varieties of $\text{GL}_2$-type.
\newblock {\em Proc. London Math. Soc.}, 81:285--317, 2000.

\bibitem{quer2}
Jordi Quer.
\newblock Fields of definition of {$\mathbb Q$}-curves.
\newblock {\em J. Th\'eor. Nombres Bordeaux}, 13(1):275--285, 2001.
\newblock 21st Journ\'ees Arithm\'etiques (Rome, 2001).

\bibitem{rib1}
Kenneth~A. Ribet.
\newblock Abelian varieties over {${\bf Q}$} and modular forms.
\newblock In {\em Algebra and topology 1992 ({T}aej\u on)}, pages 53--79. Korea
  Adv. Inst. Sci. Tech., Taej\u on, 1992.

\bibitem{ser3}
J.-P. Serre.
\newblock {\em Corps locaux}.
\newblock Publications de l'institut de math\'ematique de l'universit\'e de
  Nancago. Hermann, 1968.

\bibitem{serre}
J.-P. Serre.
\newblock {\em Cohomologie galoisienne}, volume~5 of {\em Lecture Notes in
  Mathematics}.
\newblock Springer-Verlag, Berlin, fifth edition, 1994.

\bibitem{shi2}
G.~Shimura.
\newblock On elliptic curves with complex multiplication as factors of the
  {J}acobians of modular function fields.
\newblock {\em Nagoya Math. J.}, 43(171):199--208, 1971.

\bibitem{wei}
C.~Weibel.
\newblock {\em An introduction to homological algebra}, volume~38 of {\em
  Cambridge Studies in Advanced Mathematics}.
\newblock Cambridge University Press, Cambridge, 1994.

\bibitem{wil}
A.~Wiles.
\newblock Modular elliptic curves and {F}ermat's last theorem.
\newblock {\em Ann. of Math.}, 141(3):443--551, 1995.

\bibitem{tawil}
A.~Wiles and R.~Taylor.
\newblock Ring-theoretic properties of certain {H}ecke algebras.
\newblock {\em Ann. of Math.}, 141(3):553--572, 1995.

\end{thebibliography}

\end{document}